\documentclass[a4paper,twoside,11pt]{amsart}
\pdfoutput=1 
\usepackage[english]{babel}
\usepackage[latin1]{inputenc}
\usepackage[T1]{fontenc} 
\usepackage{amsmath}
\usepackage{amssymb}
\usepackage{amsfonts}
\usepackage{mathrsfs}
\usepackage{hyperref}
\usepackage{xcolor}

\usepackage{enumerate}

\usepackage[matrix,arrow,cmtip]{xy} 
\SelectTips{cm}{}
\newdir{(}{{}*!/-5pt/@^{(}}
\newdir{(x}{{}*!/-5pt/@_{(}}
\newdir{+}{{}*!/-9pt/{}}
\newdir{>+}{@{>}*!/-9pt/{}}
\entrymodifiers={+!!<0pt,\fontdimen22\textfont2>}

\usepackage[nosubsections]{mytheorems2}

\newtheoremstyle{dtheoremnopar}{3 mm}{1 mm}{\itshape}{}{\bfseries}{.}{ }
{\thmname{#1}\thmnumber{ #2}\thmnote{ \mdseries(#3)\bfseries}}

\theoremstyle{dtheoremnopar}
\newcounter{theoremx}
\newtheorem{theoremalpha}[theoremx]{Theorem}

\newcommand{\tref}[1]{\ref{#1}} 

\newcommand{\spref}[1]{\href{http://stacks.math.columbia.edu/tag/#1}{#1}}
\newcommand{\spcite}[1]{\cite[\spref{#1}]{stacks-project}}
\newcommand\NN{\mathbb{N}}
\newcommand\QQ{\mathbb{Q}}
\newcommand\inj{\hookrightarrow}
\newcommand\surj{\twoheadrightarrow}
\DeclareMathOperator{\Mor}{Mor}
\DeclareMathOperator{\Ext}{Ext}
\DeclareMathOperator{\Ann}{Ann}
\DeclareMathOperator{\coker}{coker}
\newcommand\sF{\mathcal{F}}
\newcommand\sG{\mathcal{G}}
\newcommand\sL{\mathcal{L}}
\newcommand\sN{\mathcal{N}}
\newcommand\sO{\mathcal{O}}
\newcommand\sP{\mathcal{P}}
\newcommand\im{\mathfrak{m}}
\newcommand\PP{\mathbb{P}}           
\renewcommand\AA{\mathbb{A}}    
\DeclareMathOperator{\Spec}{Spec}
\DeclareMathOperator{\Proj}{Proj}
\newcommand\red{\mathrm{red}}   
\newcommand\cons{\mathrm{cons}} 
\newcommand\Hilb{\mathrm{Hilb}}
\newcommand\Quot{\mathrm{Quot}}
\newcommand\Set{\mathbf{Set}}
\newcommand\Sch{\mathbf{Sch}}

\DeclareMathOperator{\Flat}{Flat} 
\DeclareMathOperator{\rk}{rk}
\DeclareMathOperator{\Bl}{Bl}
\newcommand{\coho}[1]{\mathcal{H}^{{#1}}}
\newcommand{\op}{\mathrm{op}}
\newcommand{\h}{\mathrm{h}} 
\newcommand{\sh}{\mathrm{sh}} 

\begin{document}

\title[Functorial flatification]{Functorial flatification of proper morphisms}
\author{David Rydh}
\address{KTH Royal Institute of Technology\\Department of 
  Mathematics\\SE-100 44 Stockholm\\Sweden}
\email{dary@math.kth.se}
\date{Jan 9, 2025}
\thanks{Supported by the Swedish Research Council 2011-5599 and
the G\"oran Gustafsson Foundation for Research in Natural Sciences and Medicine.}

\begin{abstract}
For proper morphisms, we give a functorial flatification algorithm by blow-ups
in the spirit of Hironaka's flatification algorithm.
In characteristic zero, this gives
functorial flatification by blow-ups in smooth centers. We also give a
functorial \'etalification algorithm by Kummer blow-ups in characteristic
zero.
\end{abstract}

\maketitle


\setcounter{secnumdepth}{0}
\begin{section}{Introduction}
The celebrated flatification theorem of
Raynaud--Gruson~\cite[Thm.~5.2.2]{raynaud-gruson} states that any morphism
$f\colon X\to S$ of finite type between schemes can be flatified by a sequence
of blow-ups. This sequence depends on many choices and is highly
non-canonical. Extending this result to the case where $S$ is an algebraic
space or an algebraic stack is therefore non-trivial. This was nevertheless
accomplished for algebraic spaces by
Raynaud--Gruson~\cite[Thm.~5.7.9]{raynaud-gruson}, for stacks with finite
stabilizers in~\cite{rydh_compactification-tame-stacks} and for arbitrary
stacks in~\cite{rydh_equivariant-flatification}.

Shortly after Raynaud--Gruson's result, Hironaka gave a different proof of the
theorem when $f$ is a \emph{proper} morphism of analytic
spaces~\cite{hironaka_flattening}. Hironaka's proof is long and complicated but
the basic idea is simple. Inspired by his ideas, we give the following
\emph{functorial} flatification result:

\begin{theoremalpha}\label{T:FLATIFICATION}
Let $f\colon X\to S$ be a proper morphism of noetherian schemes. Let $U\subseteq
S$ be the largest open substack such that $f|_U$ is flat. Then there exists a
sequence of blow-ups $\widetilde{S}\to S$ with centers disjoint from $U$ such
that the strict transform $\widetilde{f}\colon \widetilde{X}\to \widetilde{S}$
is flat. Moreover, this sequence is \emph{functorial} with respect to flat
morphisms $S'\to S$ between noetherian schemes.
\end{theoremalpha}

Using functorial embedded resolution of
singularities~\cite{bierstone-milman_functoriality-in-res-of-sing}, we
immediately obtain the following smooth variant:

\begin{theoremalpha}\label{T:FLATIFICATION-SMOOTH}
Let $k$ be a field of characteristic zero and let $S$ be a smooth $k$-scheme of
finite type. Let $f\colon X\to S$ be a proper morphism. Let $U\subseteq S$ be
the largest open substack such that $f|_U$ is flat. Then there exists a sequence
of blow-ups $\widetilde{S}\to S$ with \emph{smooth} centers disjoint from $U$
such that the strict transform
$\widetilde{f}\colon \widetilde{X}\to \widetilde{S}$ is flat. Moreover, this
sequence is \emph{functorial} with respect to smooth morphisms $S'\to S$ of
finite type.
\end{theoremalpha}

By functoriality, these results immediately extend to algebraic stacks. They also
extend to proper non-representable morphisms and even to non-separated
universally closed morphisms (e.g., good moduli space morphisms).  There is
also functorial flatification of coherent sheaves of $\sO_X$-modules. See
Theorem~\ref{T:functorial-flattening} for the precise statement.

A very important use case is flatification of
\emph{modifications}, i.e., representable proper birational maps. The key point
is that a flat modification is an isomorphism. For \emph{stack-theoretic
modifications}, i.e., non-representable proper birational maps, this is no
longer the case. Prominent examples are root stacks. But an \'etale
stack-theoretic modification is an isomorphism.

To
make a flat ramified map \'etale, we need something more than blow-ups. In this
paper, we only treat \'etalification in characteristic zero and then
\emph{Kummer blow-ups} are enough. A Kummer blow-up is a blow-up followed by a
\emph{root stack} along the exceptional divisor. In addition, we only treat
the smooth case:

\begin{theoremalpha}[Functorial \'etalification of proper morphisms in
    characteristic zero]\label{T:ETALIFICATION}
  Let $S$ be a noetherian algebraic stack, smooth over a field of
  characteristic zero. Let $f\colon X\to S$ be a proper morphism with finite
  diagonal. Let $U\subseteq S$ be the largest open substack such that
  $f|_U$ is \'etale. Then there exists a commutative diagram
\[
\xymatrix{%
\widetilde{X}\ar[d]_{\widetilde{f}}\ar[r]^{q} & X\ar[d]^{f}\\
\widetilde{S}\ar[r]^{p} & S\ar@{}[ul]|\circ
}
\]
where $\widetilde{f}$ is \'etale, $p$ is a sequence of Kummer blow-ups with
smooth
centers disjoint from $U$ and $q$ is a sequence of Kummer blow-ups with
centers disjoint from $f^{-1}(U)$.  Moreover, these sequences are
\emph{functorial} with respect to smooth morphisms $S'\to S$ of finite type.
\end{theoremalpha}

\subsection*{Applications}
Our first application is that any (stack-theoretic) modification becomes a (Kummer) blow-up
after
replacing the source with a (Kummer) blow-up and this procedure is functorial
(Theorems~\ref{T:blowups-cofinal} and~\ref{T:Kummer-blowups-cofinal}).

Our second application is that the indeterminacy locus of a
birational map $f\colon X\dashrightarrow Y$ has a functorial resolution
by (Kummer) blow-ups when the target is
proper. Resolving the indeterminacy locus reduces the problem of weak
factorization of $f$ to the situation where $f$ is a sequence of
blow-ups. The latter situation can be solved functorially, resulting in
a proof of weak factorization of birational morphisms of Deligne--Mumford
stacks in characteristic zero that is completely
functorial~\cite{rydh_weak-factorization}.
It is also likely that the algorithm
in~\cite{abramovich-etal_weak-factorization} becomes functorial if
Hironaka's flatification theorem is replaced with ours.

Our third application is a general Chow lemma (Theorem~\ref{T:Chow-lemma}).

\subsection*{On the proofs}
For \emph{projective} morphisms there is an easy functorial flatification
algorithm. There is a canonical stratification of $S$ in locally closed
substacks $S_P$ indexed by Hilbert polynomials. If one blows up
$\overline{S_P}$ then the non-flatness over $S_P$ improves
(Lemma~\ref{L:improvement-after-blowup}). We obtain a functorial algorithm by
blowing up the $\overline{S_P}$ starting with the largest $P$ and continuing
in decreasing order.

In general, there is also a minimal modification that flatifies $X\to S$.
This is obtained by taking the closure of $U$ in the Hilbert scheme 
$\Hilb(X/S)$ (or the Quot scheme $\Quot(\sF/X/S)$ for coherent sheaves).
This, however, only gives a flattening modification, not a sequence
of blow-ups.

Hironaka's algorithm uses the following three key facts:
\begin{enumerate}
\item There is a canonical \emph{filtration} of $S$ in open
  subschemes such that $f$ is flat over the reduced strata. This follows from
  generic flatness since $S$ is noetherian
  (Remark~\ref{R:flattening-filtration}). In the complex-analytic setting
  this is a theorem of Frisch.
\item \'Etale-locally around any point $s\in S$, there is a maximal closed
  subscheme $Z$ of $S$ passing through $s$ with the following universal
  property: $f|_Z$ is flat and any morphism $g\colon S'\to S$ from a connected
  scheme $S'$ such that $f_{S'}$ is flat and $s\in g(S')$, factors through
  $Z$.
\item Blowing up $Z$ as in (ii) improves the non-flatness at every point of the
  inverse image of $Z$.
\end{enumerate}
Raynaud--Gruson proves (ii) for schemes using their method of
d\'evissage \cite[Thms.~4.1.2]{raynaud-gruson}. Hironaka proves (ii) and (iii)
for analytic spaces using pseudo-free presentations which is a
complex-analytic analogue of the d\'evissage
method~\cite[Thm.~2.4]{hironaka_flattening}. In our treatment, we deduce (ii)
from the existence of the universal flattening
(Proposition~\ref{P:univ-flattening-representable}) without d\'evissage. To
establish (iii), however, d\'evissage is used
(Lemma~\ref{L:improvement-after-blowup}).

The ultimate goal of both Hironaka's algorithm and our algorithm is to
make $Z$ globally defined so that it can be blown-up. We do this by
``resolving the flattening monomorphism'' (Proposition~\ref{P:resolving-mono})
and this uses the canonical flattening filtration.
Hironaka accomplishes this in~\cite[\S4]{hironaka_flattening}, also using
the canonical flattening filtration but with a more complicated
algorithm. His blow-ups are in \emph{permissible} centers $D$.
In particular, $f|_D$ is flat whereas this is essentially never the case in
our algorithm (Example~\ref{EX:algorithm}).
Hironaka also assumes that $S$ is reduced.

Functorial \'etalification is proven in
Section~\ref{S:functorial-etalification} and follows from functorial
flatification and the generalized Abhyankar lemma. 

\subsection*{Remarks}

\subsubsection*{D\'evissage and flatness}
The d\'evissage of Raynaud--Gruson is a somewhat complicated machinery.
We use d\'evissage to give a very simple characterization of flatness
(Theorem~\ref{T:devissage}) that could be of independent interest.

\subsubsection*{Functoriality of Hironaka's algorithm}
In~\cite[1.2.4]{abramovich-etal_weak-factorization}, the authors remark that
Hironaka's algorithm does not depend on any choices and is invariant under
isomorphisms but is not functorial with respect to smooth morphisms, nor under
localizations. A crucial point where it is seems that the algorithm is not
functorial with respect to localizations is when picking $\beta\geq \alpha$ to
form the center $D_\beta$~\cite[p.~542 and Lem.~3.2]{hironaka_flattening}.

Hironaka's example of a non-projective
threefold (Example~\ref{EX:Hironaka}) gives rise to a non-projective but
locally projective birational morphism $f\colon X\to Y$. The na\"ive
flatification algorithm for projective morphisms, using flattening
stratifications, does not glue to a flatification algorithm for $f$. However,
even in the projective case, Hironaka's algorithm does not equal the na\"ive
flatification algorithm (Example~\ref{EX:algorithm}).
Thus, this example does not prove that Hironaka's algorithm
is not functorial (as~\cite[1.2.4]{abramovich-etal_weak-factorization} seems
to suggest).

\subsubsection*{Functorial \'etalification for singular stacks}
In \cite{rydh_compactification-tame-stacks,rydh_equivariant-flatification},
\'etalification
is proved in characteristic zero and also in positive characteristic
under tameness assumptions.
There is no smoothness assumption but the algorithm is highly non-functorial.
It is natural to ask whether there is a functorial \'etalification algorithm
for proper maps without the smoothness assumption. Ideally, one would have
something that guides the algorithm analogously to the universal flattening.

\subsubsection*{Functorial \'etalification in positive characteristic}
The author has proved \'etalification in positive characteristic using
Artin--Schreier stacks~\cite{rydh_OWR-ramification}.
Again, this is highly non-functorial
and one could ask whether there is a functorial algorithm. This seems very
difficult.

\subsubsection*{Noetherian assumption}
This article is throughout written in the context of noetherian schemes and
stacks.  Although the majority of the paper is true in the non-noetherian
setting under suitable finiteness assumptions, there is one crucial place where
the noetherian assumption is necessary. In the noetherian setting we have a
\emph{canonical} flattening filtration
(Remark~\ref{R:flattening-filtration}). In the non-noetherian setting, we
have non-canonical flattening quasi-compact filtrations,
see~\cite[Thm.~8.3]{rydh_approximation-sheaves}. These can be obtained by
pulling-back the canonical flattening filtration of a noetherian
approximation but are not unique. This is essentially due to the fact that the
map $S\to S_0$ from a non-noetherian scheme to a noetherian approximation
$S_0$ is not flat. Similarly, we cannot deduce a non-noetherian version of
the main theorem by noetherian approximation since the result is only functorial
with respect to flat morphisms.
In fact, we give a non-noetherian example for which there cannot be a
functorial flattening algorithm by blow-ups
(Example~\ref{E:non-noetherian-counter-example}).

\subsection*{Comments}
This paper was mostly written in 2015 after reading Hironaka's paper
\cite{hironaka_flattening} with the purpose of obtaining \emph{functorial}
weak factorization of Deligne--Mumford stacks~\cite{rydh_weak-factorization}.
The paper was then finished almost ten years later.

There is a recent preprint by Michael McQuillan~\cite{mcquillan_flattening},
also inspired by Hironaka's paper, which gives a smooth-functorial
flatification similar to Theorem~\ref{T:FLATIFICATION}, also including formal
stacks.
The proofs are similar but McQuillan uses a ``minimal compactification'' of
the flattening monomorphism instead of the \'etale envelope and \'etale
d\'evissage employed in the proof of Proposition~\ref{P:resolving-mono}.
%

\subsection*{Acknowledgments}
The author would like to thank Ludvig Modin and Alessandro D'Angelo for useful
comments and Cristina Manolache for pointing out McQuillan's preprint.

\end{section}
\setcounter{secnumdepth}{3}


\begin{section}{Functorial flatification}\label{S:functorial-flatification}
Let $f\colon X\to S$ be a morphism of finite type between noetherian algebraic
stacks and let $\sF$ be a coherent $\sO_X$-module. Let $U$ be the maximal
open substack of $S$ such that $\sF|_{f^{-1}(U)}$ is $U$-flat. If $S$ is reduced,
then $U$ is dense. If $S$ is non-reduced, it may happen that $U=\emptyset$.
The \emph{$f$-torsion} of $\sF$ is the kernel of the adjunction morphism $\sF\to
j_*j^*\sF$ where $j\colon f^{-1}(U)\to X$ is the inclusion. We will frequently
let $\widetilde{\sF}$ denote $\sF$ modulo its $f$-torsion. When $S$ is a smooth
curve, then $\widetilde{\sF}$ is $S$-flat. If $S'\to S$ is a morphism, then we
let $\sF\times_S S'$ denote the pull-back of $\sF$ along $X\times_S S'\to X$.
The strict transform of $\sF$ is $\widetilde{\sF\times_S S'}$, that is, the
sheaf $\sF\times_S S'$ modulo its $f'$-torsion where $f'\colon X'\to S'$ is
the pull-back of $f$.

\begin{theorem}\label{T:functorial-flattening}
Let $S$ be a noetherian algebraic stack, let $f\colon X\to S$ be a universally
closed morphism of finite type (e.g., proper) and let $\sF$ be a coherent
$\sO_X$-module. Let $U\subseteq S$ be
the largest open substack such that $\sF|_{f^{-1}(U)}$ is $U$-flat.
Then there exist a
sequence of blow-ups $S'\to S$ with centers disjoint from $U$ such that
the
strict transform of $\sF$ is $S'$-flat. This sequence is functorial with
respect to flat morphisms $S'\to S$ between noetherian algebraic stacks.

If in addition $S$ is smooth over a field of characteristic zero, then the
centers can be taken to be smooth and this sequence is functorial with respect
to smooth morphisms.
\end{theorem}

\begin{remark}
A natural generalization of the latter part of the theorem would be to allow
$S$ to be a regular and quasi-excellent $\QQ$-stack and obtain a sequence of
blow-ups in regular centers, functorial with respect to regular morphisms.
Temkin has proven functorial desingularization of quasi-excellent
$\QQ$-stacks~\cite{temkin_functorial-desingularization-nonembedded-case,temkin_functorial-desingularization-embedded-case}
but unfortunately he has not proven strong
principalization~\cite[\S1.1.10]{temkin_functorial-desingularization-embedded-case}
which would be required to deduce this generalization.
\end{remark}

\begin{example}[Non-noetherian counter-example]\label{E:non-noetherian-counter-example}
Let $S$ be the scheme $\Spec k[x,y_1,y_2,\dots]/(xy_1,xy_2,\dots)$ which
has two irreducible
components, $V(x)$ and $V(y_1,y_2,\dots)$, meeting in a point $P$. Let
$X=V(x)\inj S$ be the inclusion of the first component which is of finite
presentation. Then $\sO_X$ is flat over $U=S\smallsetminus P$. Any functorial
algorithm is necessarily $U$-admissible. But there are no finitely generated
ideals $I$ such that $V(I)=P$, hence no $U$-admissible blow-ups.
\end{example}


\subsection{Universal flattening and flattening stratifications}
A crucial ingredient for the proof is the \emph{universal flattening}
(also called the universal flatificator) of $\sF$,
a monomorphism $\Flat_{\sF/S}\to S$ that universally ``makes $\sF$
flat''. This has appeared many times, e.g., in \cite[Thm.~2]{murre},
\cite[Thms.~4.1.2 \& 4.3.1]{raynaud-gruson},
\cite[\S3, Eqn.~13]{olsson-starr_quot-functors},
\cite{kresch_flat-strat} and \cite[\spref{05MW}, \spref{05UH}]{stacks-project}.

For $f$ and $\sF$ as in Theorem~\ref{T:functorial-flattening},
consider the functor
$\Flat_{\sF/S}\colon\Sch_{/S}^\op\to \Set$ such that $\Flat_{\sF/S}(T)$ is the
one-point set if $\sF\times_S T$ is flat over $T$ and the empty set
otherwise.
This means that a morphism $T\to S$ factors through $\Flat_{\sF/S}$
if and only if $\sF\times_S T$ is flat over $T$ and then the factorization is
unique. The functor $\Flat_{\sF/S}$ is a
\emph{bijective monomorphism} $\Flat_{\sF/S}\inj S$, that is,
\[
\Flat_{\sF/S}(T)\to \Mor_S(T,S)=\{*\}
\]
is injective for every $S$-scheme $T$, and bijective if $T$ is the
spectrum of a field.

\begin{remark}\label{R:flattening-filtration}
There is always a canonical \emph{flattening filtration}. This is a sequence of
open
subsets $U_0=\emptyset \subset U_1 \subset U_2 \subset \dots \subset U_n=|S|$
such that $\sF$ is flat over $(U_i\smallsetminus U_{i-1})_\red$.
Given $U_i$, we define $U_{i+1}$ as follows. Let
$Z_i=(S\smallsetminus U_i)_\red$ and let $V\subseteq Z_i$ be the largest
open subset such that $\sF\times_S V$ is flat over $V$. Then let
$U_{i+1}=V\cup U_i$. Since $S$ is noetherian, it has the ascending chain
condition on open subsets so the sequence is finite.

The flattening filtration gives rise to a canonical reduced \emph{flattening
stratification} $T=\coprod_{i=1}^n (U_i\smallsetminus U_{i-1})_\red$ of $S$.
\end{remark}

\begin{proposition}\label{P:univ-flattening-representable}
$\Flat_{\sF/S}\to S$ is representable and of finite type.
\end{proposition}
\begin{proof}
We give two different proofs of representability.

The first proof requires that $f$ is separated, that is, \textbf{proper}.
We have a natural map $\Flat_{\sF/S}\to \Quot(\sF/X/S)$ taking $T$ to the
quotient $\sF\times_S T \xrightarrow{\cong} \sF\times_S T$ with trivial kernel.
This is an open subfunctor: if $T$ is any $S$-scheme and
$\varphi\colon \sF\times_S T\surj \sG$ is a surjection onto a finitely
presented sheaf $\sG$ that is flat over $T$ then the subfunctor
of $S$ where $\varphi$ becomes an isomorphism is open. The Quot-functor is
represented by an algebraic space that is separated and locally of finite
presentation over $S$, for any proper morphism
$X\to S$~\cite[Thm.~1.5]{olsson_proper-coverings}. This proves that $\Flat_{\sF/S}\to S$ is representable and locally of finite presentation.

For \textbf{non-separated} $f$, the $\Quot$ functor is not representable so a
different proof is required. We will use Murre's representability theorem.
The question is local on $S$ so we can assume that $S$ is a scheme. Let
$p\colon X'\to X$ be a smooth surjective morphism from a scheme. Then
$\Flat_{\sF/S}=\Flat_{p^*\sF/S}$ so \cite[Thm.~2]{murre} applies and states
that $\Flat_{\sF/S}$ is representable if and only if $\Flat_{\sF/S}(\Spec A)\to
\varprojlim_n \Flat_{\sF/S}(\Spec A/\im^n)$ is bijective for every
complete local
noetherian ring $A$ with maximal ideal $\im$. Suppose that the right hand side
is non-empty. Then, by the local criterion of flatness, $\sF\times_S \Spec A$
is $A$-flat at every point of the special fiber, hence in an open neighborhood
of the special fiber. Since $f$ is universally closed, every open neighborhood
of the special fiber is $X\times_S \Spec A$ so $\Flat_{\sF/S}(\Spec A)$ is
non-empty.

To see that $\Flat_{\sF/S}\to S$ is quasi-compact, we consider the
flattening stratification $T\to S$ of Remark~\ref{R:flattening-filtration}.
This gives rise to a surjection $T\to \Flat_{\sF/S}$ so $\Flat_{\sF/S}$ is
quasi-compact since $T$ is a finite union of locally closed subschemes.
\end{proof}

When $f$ is \emph{projective} with a specified ample line bundle, taking the
Hilbert polynomials of the fibers $\sF|_s$ is an upper-semicontinuous function
on $S$. If $T\to S$ is a morphism with $T$ reduced, then $\sF\times_S T$ is flat
over $T$ if and only if the Hilbert polynomial is locally constant.
Thus, we have a filtration of open substacks $S_{\leq P}\subseteq S$
such that $f$ is flat over the induced reduced strata
$(S_P)_\red=S_{\leq P}\smallsetminus S_{<P}$.
Moreover, it follows that the universal
flattening is a disjoint union of locally closed substacks $S_P$ indexed by
Hilbert polynomials $P$.
%
Note that the flattening filtration $S_{\leq P}$ need not be compatible with
the canonical flattening filtration, see Example~\ref{EX:algorithm}.

The following examples show that, in general, the universal flattening
$\Flat_{\sF/S}\to S$ is not a disjoint union of locally closed subschemes,
contrary to the expectation of~\cite[Rmk~(1) after
  Thm.~3.2]{olsson-starr_quot-functors}.

\begin{example}[Hironaka~{\cite[Ex.~2]{hironaka_flattening}}]\label{EX:Hironaka}
  Let $Y$ be a complex manifold of dimension $3$ and let $C\inj Y$ be a curve
  with a single node $P$. Locally around $P$, the curve $C$ is a union of two
  smooth irreducible curves $C_1$ and $C_2$ meeting at $P$.  Let $f\colon X\to
  Y$ be the modification where we, locally around $P$, first blow up $C_1$ and
  then blow up the strict transform of $C_2$. Outside $P$, this is simply the
  blow-up of $C$.
  
  Then $\Flat_{\sO_X/Y}=(Y\smallsetminus C) \amalg (\widetilde{C}\smallsetminus
  P_2)$ where $P_2$ is the point of the normalization $\widetilde{C}$
  corresponding to the local branch $C_2$. Indeed, locally around $P$, the
  restriction to the first branch $f|_{C_1}$ is flat, whereas the restriction
  to the second branch $f|_{C_2}$ is not flat. (The first restriction is
  a flat family of $\PP^1$s degenerating to two intersecting $\PP^1$s whereas
  the second restriction is a $\PP^1$-bundle with an extra irreducible
  component over $P$.)

  The canonical reduced flattening stratification is $(Y\smallsetminus C) \amalg
  (C\smallsetminus P) \amalg P$ but it is not the universal flattening.
\end{example}

\begin{example}[Kresch~{\cite{kresch_email-aug11-2010}}]
  Let $g\colon Y\to S$ be an \'etale covering of degree 2 between projective
  smooth threefolds. Choose a curve $D$ in $S$ with a single node such that the
  preimage $g^{-1}(D)$ is the union of two smooth curves $C_1$ and $C_2$
  meeting transversally at two points $P$ and $Q$ --- the preimages of the node.

  Let $p\colon X\to Y$ be Hironaka's construction of a non-projective
  $3$-fold~\cite[B.3.4.1]{hartshorne}). That is, over $P$ we first blow up
  $C_1$ and then the strict transform of $C_2$ and over $Q$ we first blow up
  $C_2$ and then the strict transform of $C_1$. This gives a smooth
  proper $3$-dimensional scheme $X$ which is not projective.
  As in the previous example, $p|_{C_1}$ is flat except at $Q$ and $p|_{C_2}$
  is flat except at $P$. We have that $\Flat_{\sO_X/Y}=\bigl(Y\smallsetminus
  (C_1\cup C_2)\bigr) \amalg (C_1\smallsetminus Q) \amalg (C_2\smallsetminus
  P)$.

  Now, consider the composition $f\colon X\to Y\to S$. Then
  $\Flat_{\sO_X/S}=(S\smallsetminus D) \amalg (\widetilde{D}\smallsetminus R)$
  where $R$ is the point of the normalization $\widetilde{D}$ that is the
  image of $Q\in C_1$ and $P\in C_2$. In particular, the universal flattening
  is not a disjoint union of locally closed subschemes.
\end{example}

\begin{example}
In~\cite[\S4]{kresch_flat-strat}, Kresch uses partial stabilization of
families of prestable curves to construct proper morphisms of schemes whose
universal flattening is not a stratification.
\end{example}

\begin{lemma}
The canonical flattening filtration of Remark~\ref{R:flattening-filtration}
commutes with flat base change $S'\to S$.
\end{lemma}
\begin{proof}
For smooth base change, this follows directly since taking reduced subschemes
commutes with smooth base change. For general flat base change we argue as
follows.

Consider the universal flattening $F:=\Flat_{\sF/S}$. Since $F\to S$ is a
bijective monomorphism of finite type, the following properties are equivalent:
proper, finite, closed immersion, nil-immersion. Let $V$ be the largest open
subset of $S$ such that $F|_V\to V$ is finite. Then $V$ is also the largest
open subset such that $F|_{V_\red}\to V_\red$ is an isomorphism, or
equivalently, the largest open subset such that $\sF|_{V_\red}$ is flat over
$V_\red$. Since $F\to S$ is of finite presentation,
$s\in |V|$ if and only if $F\times_S \Spec(\sO_{S,s})\to \Spec(\sO_{S,s})$
is finite.

If $s'\in |S'|$ such that $F\times_S \Spec(\sO_{S',s'})\to \Spec(\sO_{S',s'})$
is finite, then by fpqc descent, $F\times_S \Spec(\sO_{S,s})\to
\Spec(\sO_{S,s})$ is also finite.
Thus the formation of $V$ commutes with flat base change.
It now follows by construction that the canonical
flattening filtration commutes with flat base change.
\end{proof}

\subsection{Flatifying modules}
In this section, we prove the main theorem when $X=S$. We thus assume that
$\sF$ is a coherent sheaf on $S$. We will later only use the case when $\sF$ is
an ideal sheaf.

For $s\colon \Spec k\to S$, let $\rk_\sF(s):=\dim_k (\sF\otimes_{\sO_S} k)$.
This only depends on $s\in |S|$. The rank function $\rk_\sF$
is upper semi-continuous and the $n$th Fitting ideal $F_n(\sF)$ cuts out the
locus where $\rk_{\sF} > n$, cf.\ \cite[\S 5.4.1]{raynaud-gruson}.

\begin{proposition}\label{P:flatify-module}
Let $S$ be a noetherian algebraic stack and $\sF$ a coherent $\sO_S$-module.
Let $\Delta\subseteq \NN$ be a subset and let $U\subseteq S$ be the largest open
substack such that $\sF|_U$ is locally free with ranks in $\Delta$. Then there
exists a single $U$-admissible blow-up $S'=\Bl_Z S\to S$ such that the strict
transform of $\sF$ is locally free with ranks in $\Delta$. Moreover,
$|Z|=S\smallsetminus U$ and $Z$ is functorial with respect to flat morphisms
$S'\to S$ between noetherian algebraic stacks.
\end{proposition}
\begin{proof}
Let $U_\delta\subseteq S$ be the largest open substack where $\sF|_U$ is
locally free of rank $\delta$. Then $U=\coprod_{\delta\in \Delta} U_\delta$ and
$U_\delta$ is empty for all but finitely many $\delta$. Let
$\overline{U_\delta}$ be the schematic closure of $U_\delta$. There are two
obstructions to the flatness of $\sF$: (i) the open substack $U$ need not be
schematically dense, and (ii) the rank could exceed $\delta$ on
$\overline{U_\delta}$. In particular, the different $\overline{U_\delta}$ could
intersect.

Let $J_\delta$ be the ideal defining $\overline{U_\delta}$. We take the center
$Z$ as the closed substack defined by the ideal
\[
I=F_0\Bigl({\textstyle \bigcap_{\delta\in \Delta} J_\delta}\Bigr) \cdot \prod_{\delta\in \Delta} \Bigl( F_\delta(\sF)+J_\delta \Bigr).
\]
Then $I|_U=\sO_U$ and $I$ commutes with flat base change since
schematic closures, finite intersections and Fitting ideals
commute with flat base change.

Let $p\colon S'=\Bl_Z S\to S$. The closed subset
$V(F_0(\bigcap_{\delta\in \Delta} J_\delta))$ equals the support of
$\bigcap_{\delta\in \Delta} J_\delta$ which is exactly the locus where $U$ is
not schematically dense. It follows that $U':=p^{-1}(U)$ is schematically dense
in $S'$. Let $\overline{U'_\delta}$ be the schematic closure of
$U'_\delta:=p^{-1}(U_\delta)$ in $S'$.

Let $\delta,\delta'\in \Delta$ and $\delta'>\delta$. On $\overline{U_\delta}\cup
\overline{U_{\delta'}}$, we have that $V(J_\delta)=\overline{U_\delta}$ and
$V(F_\delta(\sF))\supseteq \overline{U_{\delta'}}$. It follows that
$\overline{U'_\delta}$ and $\overline{U'_{\delta'}}$ are disjoint,
cf.\ \cite[Lem.~5.1.5]{raynaud-gruson}.

Let $\sF'=\sF\times_S S'$. On $\overline{U'_\delta}$, we have that
$F_\delta(\sF)\sO_{S'}=F_\delta(\sF')$ is invertible and that $\sF'$ is locally
free of rank $\delta$ on the schematically dense open subset $U'_\delta$. It
follows that the strict transform is locally free of rank
$\delta$~\cite[Lem.~5.4.3]{raynaud-gruson}.
\end{proof}

\begin{remark}
Example~\ref{E:non-noetherian-counter-example} shows that
Proposition~\ref{P:flatify-module} is false for $S$ non-noetherian even if
$\sF$ is of finite presentation. In the example, $\sF=\sO_S/(x)$,
$J_0=(y_1,y_2,\dots)$, $J_1=(x)$, $F_0(\sF)=(x)$, $F_1(\sF)=(1)$ and the
problem is that $F_0(\sF)+J_0$ is not of finite type.

Similarly, if instead $S=\Spec k[x,y_1,y_2,\dots]/(x^2,xy_1,xy_2,\dots)$ and
$\sF=\sO_S/(x)$. Then $J_0=(1)$, $J_1=(x)$, $F_0(\sF)=(x)$, $F_1(\sF)=(1)$ but
$F_0(J_0\cap J_1)=(x,y_1,y_2,\dots)$ is not of finite type.
\end{remark}

\subsection{Resolving monomorphisms}
Recall that a monomorphism $F\to S$, locally of finite type, is unramified,
that is, $\Omega_{F/S}=0$.  An
unramified morphism $h\colon F\to S$ is \'etale-locally on $F$ and $S$ a closed
immersion. In fact, there is even a canonical factorization
$h=e\circ i\colon F\inj E\to S$
where $i\colon F\inj E$ is a closed immersion and $e\colon E\to S$ is
\'etale~\cite{rydh_embeddings-of-unramified}.

The closed immersion $i$ is a regular embedding of codimension $\delta$ at
$x\in F$ if and only if $F\to S$ is a local complete intersection at $x$ and
$\coho{-1}(L_{F/S})$ is locally free of rank $\delta$ at $x$. In this situation, we
say that $F\to S$ is a local regular embedding of codimension $\delta$. If $F\to S$
is a local regular embedding at every point,
then $\coho{-1}(L_{F/S})$ is locally free and hence $F$ is a disjoint union
$\coprod_\delta F_\delta$ where $F_\delta\to S$ is a
local regular embedding of codimension $\delta$. A local regular embedding of
codimension $1$ is a ``local Cartier divisor''. The following are equivalent
for an unramified morphism $F\to S$ and a point $x\in F$:
\begin{enumerate}
\item $F\to S$ is a local regular embedding of codimension $0$ at $x$.
\item $F\to S$ is \'etale at $x$, and
\item $F\to S$ is flat (and locally of finite presentation) at $x$,
\end{enumerate}
If $h\colon F\to S$ is a monomorphism, then (i)--(iii) at $x$ are equivalent to
\begin{enumerate}\setcounter{enumi}{3}
\item $F\to S$ is an isomorphism in an open neighborhood of $h(x)$.
\end{enumerate}
If $F=\Flat_{\sF/S}$, then the largest open substack $U\subseteq S$ such that
$\sF$ is flat over $U$ coincides with the largest open substack $U\subseteq S$
such that $h|_U$ is an isomorphism.

We will now resolve the monomorphism $h\colon F:=\Flat_{\sF/S}\to S$ by which
we mean a modification $S'\to S$ such that $F\times_S S'\to S'$ is a local
regular embedding of codimension $\leq 1$.

\begin{proposition}\label{P:resolving-mono}
Let $S$ be a noetherian algebraic stack and let $h\colon F\inj S$ be a
monomorphism of finite type. Let $U\subseteq S$ be the largest open substack
such that $h|_U$ is a local regular embedding of codimension $\leq 1$.
Then there exists a $U$-admissible sequence of blow-ups $S'\to S$,
functorial with respect to flat morphisms, such that $F':=F\times_S
S'\to S'$ is a local regular embedding of codimension $\leq 1$.
\end{proposition}

Note that $U$ contains the largest open substack over which $h$ is an
isomorphism and that $F'=F'_0\amalg F'_1$ where $F'_0\to S'$ is an open
immersion and $F'_1\to S'$ is a local regular embedding of codimension $1$.

\begin{proof}[Proof of Proposition~\ref{P:resolving-mono}]
First assume that $h$ is a \textbf{closed immersion} and let
$I\subseteq \sO_S$ be the
ideal defining $h$. Then $U=U_0\amalg U_1$ where $I$ is locally free of rank $\delta$ over
$U_\delta$. By Proposition~\ref{P:flatify-module}
(with $\Delta=\{0,1\}$)
there exists a canonical $U$-admissible blow-up $S'=\Bl_Z S\to S$ that
flatifies $I$, that is, such that the strict transform of $I$ is locally
free. The strict transform of $I$ is nothing but the inverse image $I\sO_{S'}$
so $F'\to S'$ is a local regular embedding.

In \textbf{general}, we consider
the \emph{\'etale envelope} of $h$, that is, the canonical factorization of
$h$ as a closed immersion $i\colon F\inj E$ followed by an \'etale
non-separated map $e\colon E\to S$ which has a canonical section $j\colon S\to
E$ whose image is the complement of
$i(F)$~\cite{rydh_embeddings-of-unramified}.
By the special case, there is a $j(S)\cup e^{-1}(U)$-admissible blow-up
$\Bl_Z E\to E$ that makes the pull-back of $i$ regular.
This blow-up can be dominated by a
sequence of $U$-admissible blow-ups on $S$
by~\cite[Prop.~4.14(b)]{rydh_compactification-tame-stacks}. To avoid dependence
on \cite{rydh_compactification-tame-stacks} and to see that the sequence on $S$
is functorial, we will repeat the proof which also is simpler in our case.

Let $\emptyset=U_0\subset U_1 \subset U_2 \subset \dots \subset U_n=S$ be the
canonical flattening filtration of $F\to S$
(Remark~\ref{R:flattening-filtration}). In particular,
$h|_{(U_i\smallsetminus U_{i-1})_\red}$ is an isomorphism for $i=1,2,\dots,n$
so that $e^{-1}(U_i\smallsetminus U_{i-1})_\red$ is the disjoint union of
two copies of $(U_i\smallsetminus U_{i-1})_\red$.

Let $K\subseteq \sO_E$ be the ideal of the center $Z$. We will show that there
is a functorial sequence of $U$-admissible blow-ups $S'\to S$ such that
$K\sO_{E\times_S S'}$ becomes invertible. Note that $K=\sO_E$, hence invertible,
over $j(S) \cup e^{-1}(U)$.

By induction, we may assume that $K|_{e^{-1}(U_{i-1})}$ is invertible.
By Proposition~\ref{P:flatify-module}, applied to
$K|_{e^{-1}(U_i)}$ and $\Delta=\{1\}$,
there is a canonical $j(U_i)\cup e^{-1}(U_{i-1})$-admissible blow-up
$\Bl_W e^{-1}(U_i)\to e^{-1}(U_i)$ such that the inverse image of $K|_{e^{-1}(U_i)}$ becomes invertible. Since $|W|\subseteq |i(F)\cap e^{-1}(U_i)|
\smallsetminus e^{-1}(U_{i-1})$, it follows that $e|_W\colon W\to U_i$
is a closed immersion with image disjoint from $U_{i-1}$. We blow up
the scheme-theoretic closure $Q:=\overline{e(W)}$.

Since $e^{-1}(Q \cap
U_i)=W\amalg j(e(W))$, the blow-up of $Q$ will transform
$W$ into a Cartier divisor so the inverse image of $K$ becomes invertible
over $e^{-1}(U_i)$. We conclude by induction.
\end{proof}

\begin{example}\label{EX:algorithm}
Let $S$ be a smooth threefold with two smooth curves $C_1$ and $C_2$ meeting
transversally at $P$. Let $f\colon X\to S$ be the blow-up of $C_1$ followed by
the blow-up of the strict transform of $C_2$. Then the universal flattening is
$F:=\Flat_{\sO_X/S}=\bigl(S\setminus (C_1\cup C_2)\bigr) \amalg C_1 \amalg (C_2\setminus
P)$ whereas the canonical flattening filtration is $U_1=S\setminus (C_1\cup C_2)$,
$U_2=S\setminus P$, $U_3=S$. The algorithm of
Proposition~\ref{P:resolving-mono} first blows up $C_1\cup C_2$ and then blows
up a point above $P$.
The simple algorithm for projective morphisms, that blows up
$\overline{S_P}$ in decreasing order, blows up $C_1$ followed by the strict
transform of $C_2$. Hironaka's algorithm first blows up $P$ and then the
strict transform of $C_1\cup C_2$. The latter two algorithms have centers over
which $f$ is flat in contrast to the first algorithm:
$f$ is not flat over $C_1\cup C_2$.
\end{example}

\subsection{D\'evissage}\label{SS:devissage}
Let $(S,s)$ be a henselian noetherian local scheme and let $f\colon X\to S$ be
a morphism of schemes, locally of finite type, let $\sF$ be a coherent
$\sO_X$-module and let $x\in f^{-1}(s)$. Then by~\cite[Prop.~1.2.3]{raynaud-gruson} there exists an \'etale neighborhood $g\colon (U,u)\to (X,x)$,
such that $g^{-1}(\sF)$ admits a \emph{total d\'evissage at $u$}:
\[
D_i=\bigl( X_i\xrightarrow{p_i} T_i,
           \sL_i\xrightarrow{\alpha_i}\sN_i\to\sP_i,
           t_i
    \bigr)
\]
where $(T_0,t_0):=(U,u)$, $\sP_0:=g^*\sF$, and for $i=1,2,..,r$:
\begin{enumerate}
\item $X_i:=V(\Ann(\sP_{i-1}))\inj T_{i-1}$, and $t_{i-1}\in X_i$,
\item $T_i\to S$ is smooth
and affine with geometrically connected fibers,
\item $p_i\colon X_i\to T_i$ is finite, $t_i=p_i(t_{i-1})$
  and $p_i^{-1}(t_i)=\{t_{i-1}\}$ as sets,
\item $\sL_i$ is a free sheaf of finite rank on $T_i$,
\item $\sN_i:=(p_i)_*\sP_{i-1}$,
\item $\alpha_i\colon \sL_i\to \sN_i$ is a homomorphism such that
  $\alpha_i \otimes k(\tau)$ is bijective where $\tau$ is the unique
  generic point of $(T_i)_s$, and $\sP_i = \coker(\alpha_i)$.
\end{enumerate}
Finally, we also have $\sP_r=0$.

Then $\sF_x$ is $S$-flat if and only if $\alpha_i$ is injective for every
$i=1,2,\dots,r$~\cite[Cor.~2.3]{raynaud-gruson}, cf.\ \spcite{05I2}. Likewise,
if $(S_1,s_1)$ is a local scheme, $(S',s')\to (S,s)$ is a local morphism,
$X'=X\times_S S'$ and $x'$ is a point above $x$, then the pull-back of the
d\'evissage $(D_i)_{i=1,\dots,r}$ is a total d\'evissage $(D'_i)_{i=1,\dots,r}$
at $x'$ and
$\sF'_{x'}$ is $S'$-flat if and only if $\alpha'_i$ is injective for every
$i=1,2,\dots,r$.

The morphism $T_i\to S$ is \emph{pure}~\cite[Def.~3.3.3,
  Ex.~3.3.4]{raynaud-gruson}. Equivalently, $L_i:=\Gamma(T_i,\sL_i)$ is a
\emph{projective} $A$-module~\cite[Thm.~3.3.5]{raynaud-gruson}. Equivalently,
since $A$ is local, $L_i$ is a \emph{free} $A$-module~\cite{kaplansky_projective-modules}.

Let $P_i=\Gamma(T_i,P_i)$. Then we have exact sequences
\[
L_i \xrightarrow{\alpha_i} P_{i-1} \to P_i \to 0
\]
for $i=1,2,\dots,r$. Since $L_i$ are free, we can lift the $\alpha_i$ to
maps $\beta_i\colon L_i \to P_0=\Gamma(U,g^*\sF)$. If we let
$\beta\colon L_1\oplus \dots \oplus L_r \to P_0$ be the sum, then $\beta$ is
surjective. As mentioned above, $\sF_x$ is $S$-flat if and only if
all the $\alpha_i$ are injective, that is, if and only if $\beta$ is
an isomorphism. We can now summarize the situation as follows.

\begin{theorem}[Free presentation and flatness]\label{T:devissage}
Let $(S,s)=\Spec A$ be a henselian noetherian local scheme, let $X$ be an
algebraic stack, locally of finite type over $S$, let $\sF$ be a coherent
$\sO_X$-module and let $x\in |X|$ be a point over $s$. Then there exists a
smooth morphism $p\colon \Spec B\to X$, a free $A$-module $L$, usually of
infinite rank, and a surjection $\beta\colon L\to M=\Gamma(\Spec B,p^*\sF)$
such that
$\beta$ is an isomorphism if and only if $\sF$ is $S$-flat at $x$.

Moreover, if $(S',s')\to (S,s)$ is a local morphism, then the pull-back
$\beta'\colon L'\to M'$ of $\beta$ to $S'$ is an isomorphism if and only if the
pull-back $\sF'$ of $\sF$ to $X\times_S S'$ is $S'$-flat at some (or
equivalently every) point $x'$ over $x$.
In particular, $\beta|_s$ is an isomorphism.
\end{theorem}
\begin{proof}
Let $p\colon U\to X$ be a smooth presentation and pick a point $u$ above
$x$. After replacing $U$ with an \'etale neighborhood, there is a total
d\'evissage. This gives $\beta\colon L\to M$ where $L=L_1\oplus\dots \oplus L_r$
and $M=\Gamma(U,p^*\sF)$ as described above.
\end{proof}

\begin{remark}
The theorem also implies that $\sF$ is flat at $x$ if and only if $\sF$ is
flat over the image of $p$.
In the setting of complex spaces, Hironaka replaces the free $L_i$ with
\emph{pseudo-free} modules. The d\'evissage is replaced with a pseudo-free
presentation~\cite[Thm.~2.1]{hironaka_flattening}.
If $X=S=\Spec A$ is local, then Theorem~\ref{T:devissage} is elementary: we can
choose a basis $\kappa(s)^n\cong \sF|_s$ and any lift $\beta\colon A^n\surj
\Gamma(S,\sF)$ is surjective by Nakayama's lemma and gives a presentation with
the requested properties.
\end{remark}

\subsection{Flattening in the local case}\label{SS:local}
Let $(S,s)$ be a henselian noetherian local scheme, let $X$ be
an algebraic stack, locally of finite type over $S$, let $\sF$ be a coherent
$\sO_X$-module and let $x$ be a point above $s$.
Let $\Flat_{\sF/S,x}$ be the functor
from local schemes $(S',s')$ above $(S,s)$ to sets that ``makes $\sF$ flat
at $x$''. That is, if $x'\in X'=X\times_S S'$ is a point above $x$, then
$\Flat_{\sF/S,x}(S',s')$ is the one-point set if $\sF$ is $S'$-flat at $x'$
and the empty set otherwise.

\begin{proposition}\label{P:local-description-of-torsion}
The functor $\Flat_{\sF/S,x}$ is represented by a closed subscheme
$\Spec(A/I)$ of $S=\Spec A$. Let $\widetilde{\sF}$ be the quotient of $\sF$
by the $I$-torsion. If $I$ is invertible, then
$\sF\otimes \kappa(s)\surj \widetilde{\sF}\otimes \kappa(s)$ is not an
isomorphism.
\end{proposition}
\begin{proof}
Pick $\beta\colon L\to M$ as in Theorem~\ref{T:devissage}.
Then $\Flat_{\sF/s,x}$ is equivalent to the functor that ``makes $\beta$
an isomorphism''. Let $(e_\alpha)_\alpha$ be a basis of $L$.
For $g\in L$, let $g=\sum_\alpha g_\alpha e_\alpha$. Let $I\subseteq A$ be
the ideal generated by the $g_\alpha$ for every $g\in \ker \beta$. Then
$\Flat_{\sF/s,x}=V(I)$. Indeed, $I$ is the smallest ideal such that
$\ker \beta \subseteq IL$.

If $I=(a)$ and $a$ is not a zero-divisor, then for every $g\in \ker \beta$,
there is a unique $g' \in L$ such that $g=ag'$. Since
$(g'_\alpha)_{g,\alpha}=(1)$ and $A$ is local, there exists some $g$ and
$\alpha$ such that $g'_\alpha$ is a unit. For such a $g$, we have that
$g'\notin \im L$ where $\im$ denotes the maximal ideal of $A$. In particular,
$\overline{g'}\in L/\im L=M/\im M$ is non-zero whereas $ag'$ is zero in
$M$.
It follows that $\sF\otimes \kappa(s)\surj \widetilde{\sF}\otimes \kappa(s)$
has non-trivial kernel.
\end{proof}

\begin{remark}
The first part of Proposition~\ref{P:local-description-of-torsion}
is~\cite[Thm.~4.1.2]{raynaud-gruson} whereas the second part does not
seem to appear there.
In the setting of complex spaces, the analogue of
both parts of Proposition~\ref{P:local-description-of-torsion}
is~\cite[Thm.~2.4]{hironaka_flattening}.
\end{remark}

\subsection{Proof of the main theorem}

\begin{lemma}\label{L:improvement-after-blowup}
Let $S$, $f\colon X\to S$ and $\sF$ be as in
Theorem~\ref{T:functorial-flattening}. Suppose that
$\Flat_{\sF/S}=F_0\amalg F_1$ where $F_0\to S$ is an open immersion and $F_1\to
S$ is a local regular embedding of codimension $1$. Let $\widetilde{\sF}$ be
the quotient of $\sF$ by its torsion, relative to some $U\subseteq F_0$.
Then $\sF\otimes \kappa(s)\surj
\widetilde{\sF}\otimes \kappa(s)$ is not an isomorphism for every $s\in |F_1|$.
\end{lemma}
\begin{proof}
The statement
is smooth-local on $S$ so we can assume that $S=\Spec A$ is an affine scheme
and we may furthermore replace $S$ with the henselization
$\Spec \sO^{\h}_{S,s}$ at some $s\in |F_1|$ and assume that
$A$ is local henselian. Then there is a non-zero divisor $a\in A$ such
that $F_1=V(a)\amalg F_1'$ where $F_1'\to S$ factors
through $D(a)$~\cite[Thm.~18.5.11c]{egaIV}.

For every $x\in f^{-1}(s)$ we have $\Flat_{\sF/S,x}=V(I_x)$
(Proposition~\ref{P:local-description-of-torsion}) which makes $\sF$ flat in a
neighborhood of $x$. Since $f$ is universally closed and $S$ is local, we have
that $\sF$ is flat if and only if $\sF$ is flat at every point $x$ above $s$.
Thus, $V(a)=\bigcap_x V(I_x)$. Since $a$ is a non-zero divisor and $A$ is local,
there exists an $x$ such that $I_x=(a)$. The result now follows from
Proposition~\ref{P:local-description-of-torsion} since killing torsion with
respect to $U$ also kills $a$-torsion since $U\subseteq D(a)$.
\end{proof}

\begin{remark}
The results of Sections~\ref{SS:devissage} and~\ref{SS:local} are only used for
the proof of Lemma~\ref{L:improvement-after-blowup} and there is perhaps a more
elementary proof. It is not difficult to see that $\sF\otimes A/(a)\to
\widetilde{\sF}\otimes A/(a)$ is not an isomorphism and this only requires that
$\Spec A/(a)\inj F$.
\end{remark}

\begin{proof}[Proof of Theorem~\ref{T:functorial-flattening}]
We will apply the following algorithm.
\begin{enumerate}
\item Let $U\subseteq S$ be the maximal open substack such that $\sF$ is flat
  over $U$. Let $S_0=S$ and let $n=0$.
\item Let $\sF_n=\widetilde{\sF\times_S S_n}$ be the strict transform
  and let $F_n=\Flat_{\sF_n/S_n}$.
\item Apply Proposition~\ref{P:resolving-mono} to produce a blow-up sequence
  $S_{n+1}\to S_n$ such that $F_n\times_{S_n} S_{n+1}\to S_{n+1}$ is a local
  regular embedding of codimension $\leq 1$.
\item If $S_{n+1}\neq S_n$, then increase $n$ by $1$ and repeat from (ii).
\end{enumerate}
Let $Z_n\subseteq |S_n|$ be the non-flat locus of $\sF_n$. Then the $Z_n$ are
closed and $Z_{n+1}\subseteq Z_n\times_{S_n} S_{n+1}$. Moreover, for every
point $s_{n+1}\in Z_{n+1}$, the surjection $\sF_n\times_{S_n}
\kappa(s_{n+1})\surj \sF_{n+1}\times_{S_{n+1}}
\kappa(s_{n+1})$ is not an isomorphism by Lemma~\ref{L:improvement-after-blowup}.

Suppose that the algorithm does not terminate. Consider the topological space
$\widetilde{S}=\varprojlim_n |S_n|$. This is a quasi-compact space since the
space $\varprojlim_n |S_n|_\cons$ has the same underlying set but a finer
topology which is compact and Hausdorff. Here $|S_n|_\cons$ denotes the
constructible topology, where open sets are arbitrary unions of constructible
sets, which is compact and Hausdorff~\cite[Prop.~1.9.15]{egaIV}.

Since $\widetilde{S}$ is quasi-compact and $Z_n\times_{S_n} \widetilde{S}$ is a
filtered system of closed subsets, its intersection is non-empty.  Thus, there
exists an infinite sequence of closed points $s_n\in |Z_n|$, $n=0,1,2,\dots$
such that
$s_{n+1}$ lies above $s_n$. Let $k$ be an algebraic closure of $\kappa(s_0)$
and choose embeddings $\kappa(s_n)\inj k$. Then we have a sequence of
non-trivial surjections
\[
\sF_0\times_{S_0} k \surj \sF_1\times_{S_1} k \surj\dots
\]
But this is a sequence of coherent sheaves on the noetherian stack
$X\times_S k$ so cannot be infinite. This finishes the proof for general $S$.

If $S$ is smooth, every ideal on $S$ can be principalized by a sequence of
blow-ups in smooth centers that is functorial with respect to smooth morphisms,
see~\cite[Thm.~3.26]{kollar_res-of-sings}
or~\cite[Thm.~1.3]{bierstone-milman_functoriality-in-res-of-sing}. If the
flatification algorithm blows-up $I_1$, $I_2$, \dots, $I_n$, then we apply
principalization to $I_1$, then to the inverse image of $I_2$ etc. This gives the
result.
\end{proof}

\begin{remark}\label{R:pure}
We have only used that $f$ is universally closed to prove that
$\Flat_{\sF/S}$ is representable (Proposition~\ref{P:univ-flattening-representable}) and to compare it with $\Flat_{\sF/S,x}$ (proof of
Lemma~\ref{L:improvement-after-blowup}). In both situations it is enough
that $f$ is of finite type and $\sF$ is
\emph{pure}~\cite[Def.~3.3.3]{raynaud-gruson}. Nevertheless, the
algorithm does not work in this setting. The problem is that the strict
transform of a pure sheaf need not be pure as the following example shows.
\end{remark}

\begin{example}
Let $S=\AA^2=\Spec k[s,t]$ and $\overline{X}=\PP^2_S=\Proj
k[s,t,x,y,z]$. Let $X=\overline{X}\smallsetminus \{s=t=x=0\}$. Let $Z=V(s)\cup
V(xy-sz^2)\inj X$. Then it is readily verified that $Z\to S$ is pure but
killing $s$-torsion gives $\widetilde{Z}=V(xy-sz^2)$ which is not pure because
above $s=0$, we have the irreducible component $x=0$ whose fiber over the origin
is empty.
%
\end{example}

\end{section}


\begin{section}{Functorial \'etalification}\label{S:functorial-etalification}
In this section we prove Theorem~\tref{T:ETALIFICATION}. As the other
main theorems, the proof does not require that $f$ is separated, only that $f$
is universally closed, of finite type, and has quasi-finite diagonal. In
particular, $f$ is relatively Deligne--Mumford.
Since the algorithm we
construct will be functorial with respect to smooth morphisms, we can assume
that $S$ is a scheme.

We begin with blowing up $X\smallsetminus f^{-1}(U)$ with the reduced
structure.  After this blow-up, $f^{-1}(U)$ is schematically dense and will
remain so throughout the algorithm. If $U$ is not dense, we blow up the
components of $S$ that $U$ does not intersect. After this $U$ is dense in $S$.

Recall that if $S'=\Bl_W S\to S$ is a blow-up, then the strict transform $X'$
is simply the blow-up of $X$ in $f^{-1}(W)$. Since $f^{-1}(U)$ is schematically
dense, the strict transform of a $U$-admissible blow-up
is also the closure of $f^{-1}(U)$ in $X\times_S S'$.

We now apply Theorem~\tref{T:FLATIFICATION-SMOOTH} to flatify $X\to S$
by a sequence of smooth $U$-admissible blow-ups. We can thus assume that $X\to
S$ is flat. In particular $X\to S$ is now quasi-finite.

Let $\widetilde{X}$ be the normalization of $X$. Note that $\widetilde{X}\to X$
is an isomorphism over $U$. Let $V\subseteq S$ be the locus where
$\widetilde{X}\to S$ is \'etale.

Let $Z=S\smallsetminus V$ with the reduced scheme structure. Resolve $Z$, that
is, perform a sequence of $V$-admissible blow-ups $S'\to S$ with smooth centers
such that $(Z\times_S S')_\red$ becomes a divisor $D_1 \cup D_2
\cup \dots \cup D_n$ where the $D_i$ are smooth and meet with simple normal
crossings. Replace $S$, $X$ and $\widetilde{X}$ with $S'$, $X\times_S S'$ and
the normalization of $X\times_S S'$.

Let $s\in |S|$ be a point of codimension $1$. Then we can define the
ramification index at $s$ as follows. Any $x\in |\widetilde{X}|$ above $s$ is
normal of codimension $1$, so the strict local ring
$\sO^{\sh}_{\widetilde{X},x}$ is a discrete valuation ring.  Therefore
$\im_s\sO^{\sh}_{\widetilde{X},x}=(\im_x)^{e(x)}$ for some integer $e(x)\geq
1$. We let $e(s)$ be the least common multiple of the $e(x)$ for all $x$ above
$s$. Note that $e(s)=1$ if $s\in V$.

Now start with $i=1$. Then start with $r=2$. Take the $r$th root stack of all
connected components of $D_i$ with ramification index $r$.
Increase $r$ by $1$ and repeat until all components of $D_i$ with $e(s)>1$
has been rooted. Then increase $i$ by
$1$ and repeat until $i=n$. When we take a root stack $S':=S\bigl(\sqrt[r]{D}\bigr)\to S$
we also take the root stack $X\times_S S' = X\bigl(\sqrt[r]{f^{-1}(D)}\bigr)) \to X$.

If $S'\to S$ is the composition of all these root stacks, replace $S$ with $S'$
and $X$ with $X\times_S S'$ and $\widetilde{X}$ with the normalization of
$X\times_S S'$. A local analysis, using that in characteristic zero, an
extension of strictly henselian discrete valuation rings $A\to B$ with
ramification index $r$ is an $r$th Kummer extension, that is, $B=A[x]/(x^r-t)$
where $t\in A$ is a uniformizer, shows that $\widetilde{X}$ is now \'etale in
codimension $1$.

Let $Y\to \widetilde{X}$ be an \'etale presentation. By the Zariski--Nagata
purity theorem~\cite[Exp.~X, Thm.~3.1]{sga1}, \cite[Exp.~X, Thm.~3.4]{sga2},
$Y\to S$ is \'etale, hence so is $\widetilde{X}\to S$.

To finish the proof, we need to replace the normalization $n\colon
\widetilde{X}\to X$ with a sequence of blow-ups. Firstly, use
Theorem~\tref{T:FLATIFICATION} to find a functorial sequence $X'\to X$ of
$f^{-1}(U)$-admissible blow-ups that flatify $n$, that is, makes $n$ into an
isomorphism. We obtain the commutative diagram
\[
\xymatrix{%
\widetilde{X}'\ar[d]_{\cong}\ar[r] & \widetilde{X}\ar[d]^n \\
X'\ar[r] & X
}%
\]
where the bottom row is a sequence of blow-ups in centers $W_i$, $i=1,2,\dots,r$
and the top row is a sequence of blow-ups in centers $n^{-1}(W_i)$.

Secondly, use Theorem~\tref{T:FLATIFICATION-SMOOTH} to find a functorial
sequence $S'\to S$ of $U$-admissible blow-ups in smooth centers that flatify
$X'\to X\to S$. Let $X''\to X'$ be the strict transform, which also is a sequence
of blow-ups. This gives us the commutative diagram
\[
\xymatrix@R=5mm{%
X''\ar[d]\ar[r] & X'\ar[d] \\
\widetilde{X}\times_S S'\ar[d] & X\ar[d] \\
S'\ar[r] & S.
}%
\]
Since $X''\to S'$ is flat and $\widetilde{X}\times_S S'\to S'$ is \'etale,
it follows that the modification $X''\to \widetilde{X}\times_S S'$ is flat,
hence an isomorphism. We have thus obtained a sequence of blow-ups
$X''\to X'\to X$ and a sequence of blow-ups with smooth centers $S'\to S$
such that $X''\to S'$ is \'etale.
\end{section}


\begin{section}{Applications of functorial flatification}
\subsection{Cofinality of blow-ups among modifications}
Our first applications are immediate consequences of the three main theorems
since flat modifications and \'etale stack-theoretic modifications are
isomorphisms.

\begin{theorem}[Blow-ups are functorially cofinal]\label{T:blowups-cofinal}
Let $S$ be a noetherian algebraic stack and let $f\colon X\to S$ be a
modification. That is, $f$ is proper, representable, and $f|_U$ is an
isomorphism for some open substack $U\subseteq S$.
Then there exists a sequence of $U$-admissible blow-ups
$X'\to X$ such that the composition
$X'\to S$ is a sequence of $U$-admissible blow-ups.
Moreover,
\begin{enumerate}
\item The sequences are functorial with respect to flat
  base change $S'\to S$.
\item If $S$ is smooth over a field of characteristic zero, there are
  sequences of blow-ups where $X'\to S$ has smooth centers,
  and they are functorial with respect to smooth base change $S'\to S$.
\end{enumerate}
\end{theorem}

\begin{theorem}[Kummer blow-ups are functorially cofinal]\label{T:Kummer-blowups-cofinal}
Let $S$ be a smooth algebraic stack over a field of characteristic zero.
Let $f\colon X\to S$ be a stack-theoretic modification, that is,
$f$ is proper, not necessarily representable, with finite diagonal,
and $f|_U$ is an isomorphism for some open substack $U\subseteq S$.
Then there exists a sequence of $U$-admissible Kummer blow-ups
$X'\to X$ such that the composition $X'\to S$ is a sequence of
$U$-admissible Kummer
blow-ups with smooth centers. The sequences are functorial with respect to
smooth base change $S'\to S$.
\end{theorem}

\subsection{Resolution of the indeterminacy locus}
As a consequence we obtain resolution of the indeterminacy locus
by a functorial sequence of blow-ups.

\begin{theorem}\label{T:resolving-indeterminacy-loci:rep}
Let $S$ be a noetherian algebraic stack, let $X$ be a noetherian
$S$-stack, let $Y\to S$ be a proper morphism and let
$f\colon X\dashrightarrow Y$ be a rational map over $S$, defined on an
open substack $U\subseteq X$.
\begin{enumerate}
\item If $Y\to S$ is representable, then there exists a sequence of
$U$-admissible blow-ups $p\colon X'\to X$ such that the map
$f\circ p\colon X'\dashrightarrow Y$ is defined everywhere.
The sequence is functorial with respect to flat base change $X'\to X$.
\item If in addition, $X$ is smooth over a field of characteristic zero,
  then the map $p$ can be
  chosen as a sequence of blow-ups in smooth centers which is functorial with
  respect to smooth base change $X'\to X$.
\item If $Y\to S$ is not representable but relatively Deligne--Mumford,
  and $X$ is smooth over a field of characteristic zero, then there exists a
  sequence of $U$-admissible Kummer blow-ups $p\colon X'\to X$ such that the map
  $f\circ p\colon X'\dashrightarrow Y$ is defined everywhere. The sequence is
  functorial with respect to smooth base change $X'\to X$.
\end{enumerate}
\end{theorem}
\begin{proof}
Consider the proper morphism $X\times_S Y\to X$ which has a section over $U$
induced by $f$. In case (i) and (ii), the section is a closed immersion and
we let $W$ be the closure of its image. This gives us a modification $W\to X$
and Theorem~\ref{T:blowups-cofinal} gives us a functorial sequence of blow-ups
$p\colon X'\to X$ such that $f\circ p$ factors as $X'\to W\to Y$ and thus
is defined everywhere.

In case (iii), the section over $U$ is a finite morphism $U\to U\times_S Y$ and
we let $W$ be the normalization of $X\times_S Y$ in $U$. Then $W\to X\times_S
Y$ is finite because $X\times_S Y$ is of finite type over a field and $U$ is
smooth. The composition $W\to X$ is a stack-theoretic modification with finite
diagonal. We now conclude as before using
Theorem~\ref{T:Kummer-blowups-cofinal}.
\end{proof}

\begin{remark}[Birational case]\label{R:birational}
If in addition $X\to S$ is proper and $f$ is birational, then $X'\to Y$
becomes a (stack-theoretic) modification and we can apply
Theorem~\ref{T:blowups-cofinal} or~\ref{T:Kummer-blowups-cofinal} to obtain
a sequence of (Kummer) blow-ups $X''\to X'$ (not necessarily in smooth centers)
such that $X''\to Y$ becomes a sequence of (Kummer) blow-ups.

In the smooth case, we can then continue by taking a sequence of (Kummer)
blow-ups $X'''\to X''$ such that the composition $X'''\to X'$ is a sequence of
(Kummer) blow-ups in smooth centers. But we cannot simultaneously arrange so
that the blow-up sequence $X'''\to Y$ has smooth centers: this would amount to
the strong factorization conjecture.
\end{remark}

\subsection{A general Chow lemma}
Our last application is a Chow lemma. In contrast to the other results,
this is \emph{not} functorial,
since the starting point is a quasi-projective open substack and there is
neither a unique maximal quasi-projective open substack, nor a functorial
projective compactification of this open substack. To which extent there
could be a functorial Chow lemma is not clear to the author.

\begin{theorem}\label{T:Chow-lemma}
Let $S$ be a noetherian algebraic stack and let $f\colon X\to S$ be a
representable proper morphism. If there exist an open substack $U\subseteq X$
such that $f|_U$ is quasi-projective, then there exists a sequence of
$U$-admissible blow-ups $X'\to X$ such that $X'\to S$ is projective.
\end{theorem}
\begin{proof}
Let $\sL$ be an $S$-ample line bundle on $U$. There is a, non-canonical,
coherent $\sO_S$-module $\sF$
and an immersion $j\colon U\inj \PP(\sF)$ such that $\sL=j^*\sO(1)$,
cf.\ \cite[Thm.~8.6(i)]{rydh_approximation-sheaves}. The theorem follows
from Theorem~\ref{T:resolving-indeterminacy-loci:rep}(i) and
Remark~\ref{R:birational} applied to the birational map
$X\dashrightarrow \overline{U}\inj \PP(\sF)$.
\end{proof}

\begin{proposition}
Let $S$ be a noetherian algebraic stack and let 
$f\colon X\to S$ be a representable and separated morphism of finite type.
Assume that
\begin{enumerate}
\item $S$ has quasi-finite and separated diagonal, or
\item $S$ has affine stabilizers and $X$ is reduced, or
\item $S$ has affine stabilizers and the generic points of $X$ has
  linearly reductive stabilizers.
\end{enumerate}
Then there exists an open dense $U\subseteq X$ such that $f|_U$ is
quasi-projective.
\end{proposition}
\begin{proof}
(i) If $S$ has quasi-finite and separated diagonal, then so has $X$ so there
  is a quasi-finite flat presentation $p\colon X'\to X$ where $X'$ is an affine
  scheme~\cite[Thm.~7.1]{rydh_etale-devissage}.
  There is a dense open substack $U\subseteq X$ such that $p|_U$ is
  finite. The morphism $p^{-1}(U)\to S$ is quasi-affine since $p^{-1}(U)$
  is quasi-affine and $S$ has quasi-affine diagonal. It follows that $U\to S$ is
  quasi-affine~\cite[Lem.~C.1]{rydh_etale-devissage}.

(ii) By standard approximation techniques, it is enough to show that for every
  generic point $x\in |X|$, the residual gerbe $\sG_x$ admits an $S$-ample
  line bundle.
  We have a factorization $\sG_x\to \sG_s\inj S$ where $s=f(x)$. Since
  $\sG_s\inj S$ is quasi-affine~\cite[Thm.~B.2]{rydh_etale-devissage},
  it is enough to show that $\sG_x$ admits a $\sG_s$-ample line bundle.

  We have a factorization $\sG_x\to \sG_s\times_{\kappa(s)} \kappa(x)\to \sG_s$.
  Since the second map is affine, we can replace $\sG_s$ with
  $\sG_s\times_{\kappa(s)} \kappa(x)$ and assume
  that $\kappa(x)=\kappa(s)$.  Let $k/\kappa(s)$ be a finite field extension
  that neutralizes both gerbes.  Then it is enough to prove that
  $\sG_x\times_{\kappa(s)} k\to \sG_s\times_{\kappa(s)} k$ is quasi-projective
  because the norm of a $\sG_s\times_{\kappa(s)} k$-ample line bundle on
  $\sG_x\times_{\kappa(s)} k$ gives
  a $\sG_s$-ample line bundle on $\sG_x$~\cite[Prop.~6.6.1]{egaII}.

  It is thus enough
  to show that if $G$ is an algebraic group scheme over $k$ and $H$ is a closed
  subgroup, then $BH\to BG$ is quasi-projective. By a theorem of Chevalley,
  there exists a finite-dimensional $G$-representation $V$ and a
  $1$-dimensional subspace $L\subseteq V$ whose stabilizer is $H$ (as a group
  scheme)~\cite[Thm.~4.27]{milne_algebraic-groups}. This gives a
  monomorphism $BH\to [\PP(V)/G]$ which is a locally closed
  immersion~\cite[Prop.~1.65(b)]{milne_algebraic-groups}. It follows that
  $BH\to BG$ is quasi-projective.

(iii) Let $x\in |X|$ be a generic point. The intersection of all open
  neighborhoods of $x$ is a $1$-point stack $X_x$ and $(X_x)_\red=\sG_x$.
  By (ii), $\sG_x$ carries an $S$-ample line bundle $\sL$. Since $\sG_x$ has
  linearly reductive stabilizer, it has cohomological dimension zero.
  In particular, if $I$ is the ideal sheaf of $\sG_x\inj X_x$, then
  $\Ext^2(\sG_x,\sL^\vee\otimes \sL\otimes I^n/I^{n+1})=0$ for all $n$.
  Thus, the obstruction for extending the line bundle $\sL$ to $X_x$ vanishes.
  An extension to $X_x$ is also $S$-ample~\cite[Prop.~4.6.15]{egaII} and
  it spreads out to an $S$-ample line bundle on some open neighborhood of $x$.
\end{proof}

\end{section}


\bibliography{functorial-flatification}
\bibliographystyle{dary}

\end{document}